\documentclass[12pt]{amsart}
\usepackage[utf8]{inputenc}
\usepackage{amsfonts}
\usepackage{amsmath}
\usepackage{amssymb}
\usepackage{amsthm}
\usepackage{xcolor}
\usepackage[margin=1.3in]{geometry}
\usepackage{graphicx}
\usepackage[11pt]{moresize}
\usepackage{tikz}
\tikzstyle{vertex}=[circle,draw=black,fill=black,inner sep=0,minimum size=3pt,text=white,font=\footnotesize]
\usepackage{nccmath}

\newtheorem{theorem}{Theorem}
\newtheorem*{conjecture*}{Conjecture}
\newtheorem{proposition}{Proposition}[section]
\newtheorem{lemma}[proposition]{Lemma}

\newtheorem{claim}[proposition]{Claim}
\theoremstyle{remark}

\newtheorem*{remark*}{Remark}
\newtheorem{problem}[proposition]{Problem}

\newcommand{\vs}{\vspace{3mm}}

\newcommand{\R}{\mathbb{R}}
\newcommand{\Z}{\mathbb{Z}}

\newcommand{\N}{\mathbb{N}}
\newcommand{\C}{\mathbb{C}}

\newcommand{\mc}{\mathcal}

\newcommand{\ep}{\epsilon}

\newcommand{\sub}{\subseteq}

\newcommand{\modd}[1]{\text{ mod } #1}

\newcommand{\wt}{\widetilde}

\newcommand{\pos}{\text{pos}}

\DeclareMathOperator{\Real}{Re}

\title{A New Upper Bound for Separating Words}
\author{Zachary Chase}
\thanks{The author is partially supported by Ben Green's Simons Investigator Grant 376201 and gratefully acknowledges the support of the Simons Foundation.}
\address{Mathematical Institute, Andrew Wiles Building, Radcliffe Observatory Quarter, Woodstock Road, Oxford OX2 6GG, UK}
\email{zachary.chase@maths.ox.ac.uk}
\date{July 23, 2020}

\begin{document}

\begin{abstract}
We prove that for any distinct $x,y \in \{0,1\}^n$, there is a deterministic finite automaton with $\widetilde{O}(n^{1/3})$ states that accepts $x$ but not $y$. This improves Robson's 1989 upper bound of $\widetilde{O}(n^{2/5})$. 
\end{abstract}

\maketitle

\section{Introduction}

Given a positive integer $n$ and two distinct 0\hspace{.3mm}-\hspace{-.3mm}1 strings $x,y \in \{0,1\}^n$, let $f_n(x,y)$ denote the smallest positive integer $m$ such that there exists a deterministic finite automaton with $m$ states that accepts $x$ but not $y$ (of course, $f_n(x,y) = f_n(y,x)$). Define $f(n) := \max_{x \not = y \in \{0,1\}^n} f_n(x,y)$. The ``separating words problem" is to determine the asymptotic behavior of $f(n)$. An easy example \cite{sublinear} shows $f(n) = \Omega(\log n)$, which is the best lower bound known to date. Goralcik and Koubek \cite{sublinear} in 1986 proved an upper bound of $f(n) = o(n)$, and Robson \cite{robson} in 1989 proved an upper bound of $f(n) = O(n^{2/5}\log^{3/5}n)$. Despite much attempt, there has been no further improvement to the upper bound to date. 

\vs

In this paper, we improve the upper bound on the separating words problem to $f(n) = \wt{O}(n^{1/3})$.

\vspace{1mm}

\begin{theorem}\label{main}
For any distinct $x,y \in \{0,1\}^n$, there is a deterministic finite automaton with $O(n^{1/3}\log^7n)$ states that accepts $x$ but not $y$.
\end{theorem}

\vspace{1mm}

We made no effort to optimize the (power of the) logarithmic term $\log^7n$. 

\vspace{1.5mm}

\section{Definitions and Notation}

A \textit{deterministic finite automaton} (DFA) $M$ is a $4$-tuple $(Q,\delta,q_1,F)$ consisting of a finite set $Q$, a function $\delta : Q\times\{0,1\} \to Q$, an element $q_1 \in Q$, and a subset $F \sub Q$. We call elements $q \in Q$ ``states". We call $q_1$ the ``initial state" and the elements of $F$ the ``accept states". We say $M$ \textit{accepts} a string $x = x_1,\dots,x_n \in \{0,1\}^n$ if (and only if) the sequence defined by $r_1 = q_1, r_{i+1} = \delta(r_i,x_i)$ for $1 \le i \le n$, has $r_{n+1} \in F$. 

\vspace{1mm}

For a positive integer $n$, we write $[n]$ for $\{1,\dots,n\}$. We write $\sim$ as shorthand for $ = (1+o(1))$. In our inequalities, $C$ and $c$ refer to (large and small, respectively) absolute constants that sometimes change from line to line. For functions $f$ and $g$, we say $f = \wt{O}(g)$ if $|f| \le C|g|\log^C|g|$ for some constant $C$. We say a set $A \sub [n]$ is \textit{$d$-separated} if $a,a' \in A, a\not = a'$ implies $|a-a'| \ge d$. For a set $A \sub [n]$, a prime $p$, and a residue $i \in [p]_0 := \{0,\dots,p-1\}$, let $A_{i,p} = \{a \in A : a \equiv i \pmod{p}\}$. 

\vspace{1mm}

For a string $x = x_1,\dots,x_n \in \{0,1\}^n$ and a (sub)string $w = w_1,\dots,w_l \in \{0,1\}^l$, let $\pos_w(x) := \{j \in \{1,\dots,n-l+1\} : x_{j+k-1} = w_k \text{ for all } 1 \le k \le l\}$ denote the set of all (starting) positions at which $w$ occurs as a (contiguous) substring in $x$. 

\vspace{1mm}

\section{An easy $\widetilde{O}(n^{1/2})$ bound, and motivation of our argument}

In this section, we sketch an argument of an $\wt{O}(n^{1/2})$ upper bound for the separating words problem, and then how to generalize that argument to obtain $\wt{O}(n^{1/3})$. 

\vspace{1mm}

For any two distinct strings $x,y \in \{0,1\}^n$, the sets $\pos_1(x)$ and $\pos_1(y)$ are of course different. A natural way, therefore, to try to separate different strings $x,y$ is to find a small prime $p$ and a residue $i \in [p]_0$ so that $|\pos_1(x)_{i,p}| \not = |\pos_1(y)_{i,p}|$; if we can find such a $p$ and $i$, then since\footnote{We make use of the fact that $q \mid a-b$ for all primes $q$ in a set $\mc{Q}$ implies $\prod_{q \in \mc{Q}} q \mid a-b$, along with standard estimates on $\prod_{q \in \mc{Q}} q$ for $\mc{Q} = \{q \le k : q \text{ prime}\}$.} there will be a prime $q$ of size $q = O(\log n)$ with $|\pos_1(x)_{i,p}| \not \equiv |\pos_1(y)_{i,p}| \pmod{q}$, there will be a deterministic finite automaton with $2pq = O(p\log n)$ states that accepts one string but not the other (see Lemma \ref{machine}). We are thus led to the following (purely number-theoretic) problem. 

\begin{problem}\label{problem}
For given $n$, determine the minimum $k$ such that for any distinct $A,B \sub [n]$, there is some prime $p \le k$ and some $i \in [p]_0$ for which $|A_{i,p}| \not = |B_{i,p}|$.
\end{problem}

Problem \ref{problem} has been considered in \cite{robson2}, \cite{scott}, and \cite{modular}\footnote{In the last reference, they look for an \textit{integer} $m \le k$ and some $i \in [m]_0$ for which $|A_{i,m}| \not = |B_{i,m}|$, which is of course more economical. We decided to restrict to primes for aesthetic reasons.} (and possibly other places) and was essentially solved in each. We present a simple solution, also discovered in \cite{modular}.

\begin{claim}
For any distinct $A,B \sub [n]$, there is some prime $p = O(\sqrt{n\log n})$ and some $i \in [p]_0$ for which $|A_{i,p}| \not = |B_{i,p}|$. 
\end{claim}

\begin{proof}
(Sketch) Fix distinct $A,B \sub [n]$. Suppose $k$ is such that $|A_{i,p}| = |B_{i,p}|$ for all primes $p \le k$ and all $i \in [p]_0$. For a prime $p$, let $\Phi_p(x)$ denote the $p^{\text{th}}$ cyclotomic polynomial, of degree $p-1$. Then since $\sum_{j=1}^n 1_A(j)e^{2\pi i \frac{aj}{p}} = \sum_{j=1}^n 1_B(j)e^{2\pi i \frac{aj}{p}}$ for all $p \le k$ and all $a \in [p]_0$, the polynomials $\Phi_p(x)$, for $p \le k$, divide $\sum_{j=1}^n (1_A(j)-1_B(j))x^j =: f(x)$. Therefore, $\prod_{p \le k} \Phi_p(x)$ divides $f(x)$. Since $A \not = B$, $f$ is not identically $0$ and thus must have degree at least $\sum_{p \le k} (p-1) \sim \frac{1}{2}\frac{k^2}{\log k}$. Since the degree of $f$ is trivially at most $n$, we must have $(1+o(1))\frac{1}{2}\frac{k^2}{\log k} \le n$.
\end{proof}

\vspace{-1mm}

By a standard pigeonhole argument (see Section $7$), the bound $\widetilde{O}(\sqrt{n})$ is sharp. 

\vspace{2mm}

A natural idea to improve this $\widetilde{O}(\sqrt{n})$ bound for the separating words problem is to consider the sets $\pos_w(x)$ and $\pos_w(y)$ for longer $w$. The length of $w$ is actually not important in terms of its ``cost" to the number of states needed, just as long as it is at most $p$, where we will be considering $|\pos_w(x)_{i,p}|$ and $|\pos_w(y)_{i,p}|$ (see Lemma \ref{machine}). One immediate benefit of considering longer $w$ is that the sets $\pos_w(x)$ and $\pos_w(y)$ are \textit{smaller} than $\pos_1(x)$ and $\pos_1(y)$; indeed, for example, it can be shown without much difficulty that for any distinct $x,y \in \{0,1\}^n$, there is some $w$ of length $n^{1/3}$ such that $\pos_w(x)$ and $\pos_w(y)$ are distinct sets of size at most $n^{2/3}$. Thus, to get a bound of $\widetilde{O}(n^{1/3})$ on the separating words problem, it suffices to show the following.

\begin{problem}\label{sparse}
For any distinct $A,B \sub [n]$ of sizes $|A|,|B| \le n^{2/3}$, there is some prime $p = \widetilde{O}(n^{1/3})$ and some $i \in [p]_0$ so that $|A_{i,p}| \not = |B_{i,p}|$. 
\end{problem} 

As in the proof sketch above, this problem is equivalent to a statement about a product of cyclotomic polynomials dividing a sparse polynomial of small degree (see the last page of \cite{modular}). We were not able to solve Problem \ref{sparse}. However, we make the additional observation that we can take $w$ so that $\pos_w(x)$ and $\pos_w(y)$ are \textit{well-separated} sets. Indeed, if $w$ has length $2n^{1/3}$ and has no period of length at most $n^{1/3}$, then $\pos_w(x)$ and $\pos_w(y)$ are $n^{1/3}$-separated sets. As we'll use later, Lemmas $1$ and $2$ of \cite{robson} show that such $w$ are common enough to ensure there is a choice with $\pos_w(x) \not = \pos_w(y)$. Our main technical theorem is thus the following\footnote{See page 4 for a more specific formulation.}.

\vspace{1mm}

\begin{theorem}\label{diffprime}
Let $A,B$ be distinct subsets of $[n]$ that are each $n^{1/3}$-separated. Then there is some prime $p = \wt{O}(n^{1/3})$ and some $i \in [p]_0$ so that $|A_{i,p}| \not = |B_{i,p}|$. 
\end{theorem} 

\vspace{1mm}

Although Theorem \ref{diffprime} is also equivalent to a question about a product of cyclotomic polynomials dividing a certain type of polynomial, we were not able to make progress through number theoretic arguments. Rather, we reverse the argument of Scott \cite{scott}, by noting that if there is some small $m$ so that the $m^{\text{th}}$-moments of $A$ and $B$ differ, i.e. $\sum_{a \in A} a^m \not = \sum_{b \in B} b^m$, then there is some small $p$ and some $i \in [p]_0$ so that $|A_{i,p}| \not \equiv |B_{i,p}| \modd p$ (and thus $|A_{i,p}| \not = |B_{i,p}|$).\footnote{The implication just written is actually quite obvious (see the deduction of Theorem \ref{diffprime} from Proposition \ref{diffmoment}); the implication of Scott, however, that some small $p$ and some $i \in [p]_0$ with $|A_{i,p}| \not \equiv |B_{i,p}| \pmod{p}$ implies the existence of some small $m$ with $\sum_{a \in A} a^m \not = \sum_{b \in B} b^m$ is less trivial, though basically just follows from the fact that $1_{x \equiv i \pmod{p}} \equiv 1-(x-i)^{p-1} \pmod{p}$.} 

\vspace{1mm}

The benefit of considering the ``moments" problem is that it is more susceptible to complex analytic techniques. Borwein, Erdélyi, and Kós \cite{littlewood} use complex analytic techniques to show that for any distinct $A,B \sub [n]$, there is some $m \le C\sqrt{n}$ with $\sum_{a \in A} a^m \not = \sum_{b \in B} b^m$. One proof of theirs was to show that any polynomial $p$ of degree $n$ with $|p(0)| = 1$ and coefficients bounded by $1$ in absolute value must be at least $\exp(-C\sqrt{n})$ at some point close to $1$. We were able to adapt this proof to find a small(er) $m$ such that $\sum_{a \in A} a^m \not = \sum_{b \in B} b^m$ in the case that $A,B$ are well-separated sets, and thus prove Theorem \ref{diffprime}. 

\vspace{1mm}

The adaptations we make are quite significant. See Lemma \ref{hproperties} and Lemma \ref{product}.

\vspace{1.5mm}

\section{Proof of Theorem \ref{main}}

In this section, we quickly deduce Theorem \ref{main} from our main number-theoretic theorem which we prove in Section 5. Recall we say $A \sub [n]$ is \textit{d-separated} if $|a-a'| \ge d$ for any distinct $a,a' \in A$.

\vspace{1.5mm}

\setcounter{theorem}{1}
\begin{theorem}\label{diffprime}
Let $A,B$ be distinct subsets of $[n]$ that are each $n^{1/3}$-separated. Then there is some prime $p \in [\frac{1}{2}C'n^{1/3}\log^6n,C'n^{1/3}\log^6n]$ and some $i \in [p]_0$ so that $|A_{i,p}| \not = |B_{i,p}|$. Here, $C' > 0$ is an absolute constant.
\end{theorem} 

\vspace{1.5mm}

Recall that, for a string $x = x_1,\dots,x_n \in \{0,1\}^n$ and a (sub)string $w = w_1,\dots,w_l \in \{0,1\}^l$, we defined $\pos_w(x) := \{j \in \{1,\dots,n-l+1\} : x_{j+k-1} = w_k \text{ for all } 1 \le k \le l\}$. 

\vspace{1mm}

\begin{lemma}\label{machine}
Let $m,n$ be positive integers, $i \in [m]_0$ a residue mod $m$, $q$ a prime number, $a \in [q]_0$ a residue mod $q$, and $w \in \{0,1\}^l$ a string of length $l \le m$. Then there is a determinsitic finite automaton with $2mq$ states that, for any string $x \in \{0,1\}^n$, accepts $x$ if and only if $|\{j \in \emph{\text{pos}}_w(x) : j \equiv i \pmod{m}\}| \equiv a \pmod{q}$.
\end{lemma}

\begin{proof}
Write $w = w_1,\dots,w_l$. We assume $l > 1$; a minor modification to the following yields the result for $l=1$. We interpret indices of $w$ mod $m$, which we may, since $l \le m$. Let the states of the DFA be $\Z_m\times\{0,1\}\times\Z_q$. The initial state is $(1,0,0)$. If $j \not \equiv i \pmod{m}$ and $\epsilon \in \{0,1\}$, set $\delta((j,0,s),\epsilon) = (j+1,0,s)$. If $j \equiv i \pmod{m}$, set $\delta((j,0,s),w_1) = (j+1,1,s)$ and $\delta((j,0,s),1-w_1) = (j+1,0,s)$. If $j \not \equiv i+l-1 \pmod{m}$, set $\delta((j,1,s),w_{j-i+1}) = (j+1,1,s)$ and $\delta((j,1,s),1-w_{j-i+1}) = (j+1,0,s)$. Finally, if $j \equiv i+l-1 \pmod{m}$, set $\delta((j,1,s),w_l) = (j+1,0,s+1)$ and $\delta((j,1,s),1-w_l) = (j+1,0,s)$. The set of accept states is $\Z_m\times\{0,1\}\times\{a\}$. 
\end{proof}

\vspace{1.5mm}

\setcounter{theorem}{0}
\begin{theorem}\label{main}
For any distinct $x,y \in \{0,1\}^n$, there is a deterministic finite automaton with $O(n^{1/3}\log^7n)$ states that accepts $x$ but not $y$.
\end{theorem}

\begin{proof}
Let $x_1,\dots,x_n$ and $y_1,\dots,y_n$ be two distinct strings in $\{0,1\}^n$. If $x_k \not = y_k$ for some $k < 2n^{1/3}$, then we are done\footnote{Simply use a DFA on $2n^{1/3}$ states that accepts exactly those strings starting with $x_1,\dots,x_{2n^{1/3}}$.}, so we may suppose otherwise. Let $k \ge 2n^{1/3}$ be the first index with $x_k \not = y_k$. Let $w' = x_{k-2n^{1/3}+1},\dots,x_{k-1}$ be a (common sub)string of $x$ and $y$ of length $2n^{1/3}-1$. By Lemma $1$ and Lemma $2$ of \cite{robson}, there is some choice $w \in \{w'0,w'1\}$ for which $A := \pos_w(x)$ is $n^{1/3}$-separated and $B := \pos_w(y)$ is $n^{1/3}$-separated. By the choice of $k$, we have $A \not = B$, so Theorem \ref{diffprime} implies there is some prime $p \in [\frac{1}{2}C'n^{1/3}\log^6n,C'n^{1/3}\log^6n]$ and some $i \in [p]_0$ for which $|A_{i,p}| \not = |B_{i,p}|$. Since $|A_{i,p}|$ and $|B_{i,p}|$ are at most $n$, there is some prime $q = O(\log n)$ for which $|A_{i,p}| \not \equiv |B_{i,p}| \pmod{q}$. Since $|w| = 2n^{1/3} \le p$, by Lemma \ref{machine} there is a deterministic finite automaton with $2pq = O(n^{1/3}\log^7n)$ states that accepts $x$ but not $y$. 
\end{proof}

\vspace{1.5mm}

\section{Proof of Theorem \ref{diffprime}}

In this section, we deduce Theorem \ref{diffprime} from the following complex analytic theorem, which we prove in Section 6.

\vs

Let $\mathcal{P}_n$ denote the collection of all polynomials $p(x) = 1-\sigma x^d+\sum_{j=n^{1/3}}^n a_jx^j \in \C[x]$ such that $1 \le d < n^{1/3}$, $\sigma \in \{0,1\}$, and $|a_j| \le 1$ for each $j$.

\vspace{1.5mm}

\setcounter{theorem}{2}
\begin{theorem}\label{largevalue}
There is some absolute constant $C_1 > 0$ so that for all $n \ge 2$ and all $p \in \mathcal{P}_n$, it holds that $\max_{x \in [1-n^{-2/3},1]} |p(x)| \ge \exp(-C_1n^{1/3}\log^5n)$. 
\end{theorem}

\vspace{1.5mm}

The deduction of Theorem \ref{diffprime} from Theorem \ref{largevalue} follows from first showing the polynomial $p(x) := \sum_{n \in A} x^n - \sum_{n \in B} x^n$ cannot be divisible by a large power of $x-1$. We will use part of Lemma 5.4 of \cite{littlewood}, stated below.

\vspace{1.5mm}

\begin{lemma}\label{order}
Suppose the polynomial $f(x) = \sum_{j=0}^n a_j x^j \in \C[x]$ has $|a_j| \le 1$ for each $j$. If $(x-1)^k$ divides $f(x)$, then $\max_{1-\frac{k}{9n} \le x \le 1} |f(x)| \le (n+1)(\frac{e}{9})^k$.
\end{lemma}

\vspace{0.25mm}

\begin{proposition}\label{maxorder}
There exists an absolute constant $C > 0$ so that for all $n \ge 1$ and all $p(x) \in \mathcal{P}_n$, the polynomial $(x-1)^{\lfloor Cn^{1/3}\log^5n \rfloor}$ does not divide $p(x)$. 
\end{proposition}

\begin{proof}
Take $C > 0$ large. Take $p(x) \in \mathcal{P}_n$. Suppose for the sake of contradiction that $(x-1)^{Cn^{1/3}\log^5n}$ divided $p(x)$. Then, by Lemma \ref{order} and Theorem \ref{largevalue}, \begin{align*}(n+1)(\frac{e}{9})^{Cn^{1/3}\log^5n} &\ge \max_{x \in [1-\frac{C}{9}n^{-2/3}\log^5n,1]} |p(x)| \\ &\ge \max_{x \in [1-n^{-2/3},1]} |p(x)| \\ &\ge e^{-C_1n^{1/3}\log^5n}, \end{align*} which is a contradiction if $C$ is large enough.
\end{proof}

\vspace{0.5mm}

We now exploit the (well-known) equivalence between common moments and a large vanishing of the associated polynomial at $x=1$. 

\begin{proposition}\label{diffmoment}
Let $A,B$ be distinct subsets of $[n]$ that are each $n^{1/3}$-separated. Then there is some non-negative integer $m = O(n^{1/3}\log^5n)$ such that $\sum_{a \in A} a^m \not = \sum_{b \in B} b^m$. 
\end{proposition}

\begin{proof}
Let $f(x) = \sum_{j=0}^n \epsilon_j x^j$, where $\epsilon_j := 1_A(j)-1_B(j)$. Let $\tilde{f}(x) = \frac{f(x)}{x^r}$, where $r$ is maximal with respect to $\epsilon_0,\dots,\epsilon_{r-1} = 0$. We may assume without loss of generality that $\tilde{f}(0) = 1$. Then the fact that $A,B$ are $n^{1/3}$-separated implies $\tilde{f}(x) \in \mathcal{P}_n$. By Proposition \ref{maxorder}, $(x-1)^{Cn^{1/3}\log^5n}$ does not divide $\tilde{f}(x)$ and thus does not divide $f(x)$. This means that there is some non-negative integer $k \le Cn^{1/3}\log^5n-1$ so that $f^{(k)}(1) \not = 0$. Take a minimal such $k$. If $k = 0$, we're of course done. Otherwise, since $f^{(m)}(1) = \sum_{j=0}^n j(j-1)\dots(j-m+1)\ep_j$ for $m \ge 1$, it's easy to inductively see that $\sum_{j \in A} j^m = \sum_{j \in B} j^m$ for all $0 \le m \le k-1$ and then $\sum_{j \in A} j^k \not = \sum_{j \in B} j^k$.
\end{proof}

\vspace{0.5mm}

We can now deduce Theorem \ref{diffprime}. 

\vspace{1.5mm}

\setcounter{theorem}{1}
\begin{theorem}\label{diffprime}
Let $A,B$ be distinct subsets of $[n]$ that are each $n^{1/3}$-separated. Then there is some prime $p \in [\frac{1}{2}C'n^{1/3}\log^6n,C'n^{1/3}\log^6n]$ and some $i \in [p]_0$ so that $|A_{i,p}| \not = |B_{i,p}|$. Here, $C' > 0$ is an absolute constant.
\end{theorem}

\begin{proof}
By Proposition \ref{diffmoment}, take $m = O(n^{1/3}\log^5n)$ such that $\sum_{a \in A} a^m \not = \sum_{b \in B} b^m$. Since $\left|\sum_{a \in A} a^m - \sum_{b \in B} b^m\right| \le n\hspace{.5mm} n^m \le \exp(O(n^{1/3}\log^6n))$, there is some prime $p \in [\frac{1}{2}C'n^{1/3}\log^6n,C'n^{1/3}\log^6n]$ such that $\sum_{a \in A} a^m \not \equiv \sum_{b \in B} b^m \pmod{p}$. Noting that $\sum_{a \in A} a^m \equiv \sum_{i=0}^{p-1} |A_{i,p}|i^m \pmod{p}$ and $\sum_{b \in B} b^m \equiv \sum_{i=0}^{p-1} |B_{i,p}| i^m \pmod{p}$, we see that there is some $i \in [p]_0$ for which $|A_{i,p}| \not \equiv |B_{i,p}| \pmod{p}$. 
\end{proof}

\vspace{1.5mm}

\section{Proof of Theorem \ref{largevalue}}

In this section, we finish off the proof of Theorem \ref{main} by proving the needed theorem about sparse Littlewood polynomials being ``large" somewhere near $1$.

\vspace{1mm}

Recall that $\mathcal{P}_n$ denotes the collection of all polynomials $p(x) = 1-\sigma x^d+\sum_{j=n^{1/3}}^n a_jx^j$ in $\C[x]$ such that $1 \le d < n^{1/3}$, $\sigma \in \{0,1\}$, and $|a_j| \le 1$ for each $j$.

\vspace{1.5mm}

\begin{theorem}\label{largevalue}
There is some absolute constant $C_1 > 0$ so that for all $n \ge 2$ and all $p \in \mathcal{P}_n$, it holds that $\max_{x \in [1-n^{-2/3},1]} |p(x)| \ge \exp(-C_1n^{1/3}\log^5n)$. 
\end{theorem}

\vspace{1mm}

For $a > 0$, define $\widetilde{E}_a$ to be the ellipse with foci at $1-a$ and $1-a+\frac{1}{4}a$ and with major axis $[1-a-\frac{a}{32},1-a+\frac{9a}{32}]$. We borrow\footnote{They state Lemma \ref{hadamard} for $p \in \mathcal{S}$, where they define $\mathcal{S}$ to be the set of all analytic functions $f$ on the (open) unit disk such that $|f(z)| \le \frac{1}{1-|z|}$ for each $z \in \mathbb{D}$. It is clear $\mathcal{P}_n \subseteq \mathcal{S}$ for each $n$.} Corollary 5.3 from \cite{littlewood}: 

\begin{lemma}\label{hadamard}
For every $n \ge 1$, $p \in \mathcal{P}_n$, and $a > 0$, we have $\left(\max_{z \in \widetilde{E}_a} |p(z)|\right)^2 \le \frac{64}{39a}\max_{x \in [1-a,1]} |p(x)|$. 
\end{lemma}

\vspace{3mm}

By Lemma \ref{hadamard}, in order to prove Theorem \ref{largevalue} it suffices to show: 

\vspace{1.5mm}

\begin{proposition}\label{mmp}
There is an absolute constant $C > 0$ so that for every $n \ge 1$ and every $p \in \mathcal{P}_n$, it holds that $\left(\max_{z \in \widetilde{E}_{n^{-2/3}}} |p(z)|\right)^2 \ge \exp(-Cn^{1/3}\log^5n)$. 
\end{proposition}

\vspace{1mm}

While \cite{littlewood} certainly uses that $\wt{E}_a$ is an ellipse, all we will use is about $\wt{E}_a$ (besides using Lemma \ref{hadamard} as a black box) is that the interior of $\wt{E}_a$, denoted $\wt{E}_a^\circ$, contains a ball of radius $\frac{a}{10^{10}}$ centered at $1-a$. We begin with two lemmas. 

\vspace{2mm}

In the proof of Theorem 5.1 of \cite{littlewood}, the authors use the function $h(z) = (1-a)\frac{z+z^2}{2}$ for a maximum modulus principle argument to lower bound the quantity $\left(\max_{z \in \wt{E}_a} |p(z)|\right)^2$. For $z = e^{2\pi i t}$ for small $t$, the magnitude $|h(e^{2\pi i t})|$ is quadratically in $t$ less than $1$. For our purposes, we need a linear deviation of $|h(e^{2\pi i t})|$ from $1$. This motivates the following lemma.

\vspace{1mm}

\begin{lemma}\label{hproperties}
There are absolute constants $c_4,c_5,C_6 > 0$ such that the following holds for $a > 0$ small enough. Let $\tilde{h}(z) = \sum_{j=1}^r d_jz^j$ for $$d_j := \frac{\lambda_a}{j^2\log^2(j+3)}$$ and $r := a^{-1}$, where $\lambda_a \in (1,2)$ is such that $\sum_{j=1}^r d_j = 1$. Let $h(z) = (1-a)\tilde{h}(z)$. Then $h(0) = 0$, $|h(e^{2\pi it})| \le 1-a$ for each $t$, $h(e^{2\pi it}) \in \widetilde{E}_a^\circ$ for $t \in [-c_4a,c_4a]$, and $$|h(e^{2\pi it})| \le 1-c_5\frac{|t|}{\log^2 (a^{-1})}$$ for $t \in [-\frac{1}{2},\frac{1}{2}]\setminus [-C_6a,C_6a]$. 
\end{lemma} 

\begin{proof}
Clearly $h(0) = 0$ and $|h(e^{2\pi it})| \le 1-a$ for each $t$. Now, for any $t \in \R$, $$|\tilde{h}(e^{2\pi i t})-1| = \left|\sum_{j=1}^r d_j(e^{2\pi i tj}-1)\right| \le \sum_{j=1}^r d_j2\pi tj = 2\pi t \sum_{j=1}^r \frac{\lambda_a}{j\log^2(j+3)} \le C_4t$$ for $C_4$ absolute. Thus, $$|h(e^{2\pi i t})-(1-a)| = (1-a)|\tilde{h}(e^{2\pi i t})-1| \le C_4t.$$ If $|t| \le c_4 a$ for $c_4 > 0$ sufficiently small, we conclude $h(e^{2\pi i t}) \in \wt{E}_a^\circ$.     

\vspace{2mm}

\noindent We now go on to showing the last inequality in the statement of Lemma \ref{hproperties}. 

\vspace{2mm}

\noindent By summation by parts, for any $z \in \mathbb{C}$, we have 
\begin{equation}\label{sumbyparts} 
\sum_{j=1}^r \frac{\lambda_az^j}{j^2\log^2(j+3)} = \frac{\lambda_a\sum_{j=1}^r z^j}{r^2\log^2(r+3)}+2\lambda_a\int_1^r \frac{(\sum_{j \le x} z^j)\left(\log(x+3)+\frac{x}{x+3}\right)}{x^3\log^3(x+3)}dx.
\end{equation} 
Quickly note that, for $z = 1$, \eqref{sumbyparts} gives 
\begin{equation}\label{sumbyparts1}
1 = \frac{\lambda_a}{r\log^2(r+3)}+2\lambda_a\int_1^r \frac{\lfloor x\rfloor\left(\log(x+3)+\frac{x}{x+3}\right)}{x^3\log^3(x+3)}dx.
\end{equation}
Trivially, for any $z \in \partial \mathbb{D}$, we have 
\begin{equation}\label{sumbypartsfirstterm}
\left|\frac{\lambda_a\sum_{j=1}^r z^j}{r^2\log^2(r+3)}\right| \le \frac{\lambda_a}{r\log^2(r+3)}.
\end{equation}
Note that, for any $x \ge 1$,
\begin{equation}\label{geomseriesbound}
\left|\sum_{j \le x} z^j\right| = \left|z\frac{1-z^{\lfloor x \rfloor}}{1-z}\right| \le \frac{2}{|1-z|} \le t^{-1}
\end{equation}
for all $z = e^{2\pi it}$ with $t \in (0,\frac{1}{2}]$. Take $C_6 > 3$ to be chosen later. Note $t \in (C_6a,\frac{1}{2}]$ implies $3t^{-1} < r$. For $z = e^{2\pi it}$ with $C_6a < t \le \frac{1}{2}$, \eqref{geomseriesbound} and \eqref{sumbyparts1} imply $$\left|2\lambda_a\int_1^r \frac{(\sum_{j \le x} z^j)\left(\log(x+3)+\frac{x}{x+3}\right)}{x^3\log^3(x+3)}dx \right| \le $$ $$2\lambda_a \int_1^{3t^{-1}} \frac{\lfloor x \rfloor\left(\log(x+3)+\frac{x}{x+3}\right)}{x^3\log^3(x+3)}dx+ 2\lambda_a\int_{3t^{-1}}^r \frac{t^{-1}\left(\log(x+3)+\frac{x}{x+3}\right)}{x^3\log^3(x+3)}dx$$
\begin{equation}\label{sumbypartsupperbound} 
= 1-2\lambda_a\int_{3t^{-1}}^r \frac{\left(\lfloor x \rfloor - t^{-1}\right)\cdot\left(\log(x+3)+\frac{x}{x+3}\right)}{x^3\log^3(x+3)}dx-\frac{\lambda_a}{r\log^2(r+3)}.
\end{equation} 
Observe $\lfloor x \rfloor -  t^{-1} \ge \frac{1}{2}x$ for $x \ge 3t^{-1}$. Therefore, \begin{align*}2\lambda_a\int_{3t^{-1}}^r \frac{\left(\lfloor x \rfloor - t^{-1}\right)\cdot\left(\log(x+3)+\frac{x}{x+3}\right)}{x^3\log^3(x+3)}dx &\ge \lambda_a\int_{3t^{-1}}^r \frac{1}{x^2\log^2(x+3)}dx \\ &\ge \frac{\lambda_a}{\log^2(r+3)}\int_{3t^{-1}}^r \frac{1}{x^2}dx\end{align*} 
\begin{equation}\label{sumbypartslowerbound} 
\hspace{81.8mm} = \frac{\lambda_at}{3\log^2(r+3)}-\frac{\lambda_a}{r\log^2(r+3)}.
\end{equation}
Combining \eqref{sumbyparts}, \eqref{sumbypartsfirstterm}, \eqref{sumbypartsupperbound}, and \eqref{sumbypartslowerbound}, we conclude that, for any $t \in (C_6a,\frac{1}{2}]$,
\begin{equation}\label{hboundlarget}
\left|\tilde{h}(e^{2\pi it})\right| = \left|\sum_{j=1}^r \frac{\lambda_ae^{2\pi i jt}}{j^2\log^2(j+3)}\right| \le 1-\frac{\lambda_at}{3\log^2(r+3)}+\frac{\lambda_a}{r\log^2(r+3)}.
\end{equation}
Taking $C_6$ to be much larger than $3$, \eqref{hboundlarget} gives the bound $$|\tilde{h}(e^{2\pi it})| \le 1-c_5\frac{t}{\log^2(a^{-1})}$$ for $t \in (C_6a,\frac{1}{2}]$, for suitable $c_5 > 0$. By symmetry, the proof is complete. 
\end{proof}

\vspace{2mm}

We from now on fix some $n \ge 1$ and some $p \in \mc{P}_n$ (defined at the beginning of the section). Let $\tilde{p}$ be the truncation of $p$ to terms of degree less than $n^{1/3}$; either $\tilde{p} = 1$ or $\tilde{p} = 1-x^d$ for some $1 \le d < n^{1/3}$. Take $a = n^{-2/3}$, and let $h$ be as in Lemma \ref{hproperties}. Let $m = c_4^{-1}n^{2/3}$. Let $J_1 =  c_5^{-1}n^{-1/3}m\log^4 n$ and $J_2 = m-J_1$. 

\vspace{2mm}

In the proof below of Proposition \ref{mmp}, we will need to upper bound the product $\prod_{j=J_1}^{J_2-1} |\tilde{p}(h(e^{2\pi i \frac{j}{m}}))|$ by $\exp(\wt{O}(n^{1/3}))$. We must be careful in doing so, as the trivial upper bound on each term is $2$ and there are approximately $n^{2/3}$ terms. However, we expect the argument of $h(e^{2\pi i \frac{j}{m}})$ to behave as if it were random, and thus we expect $|\tilde{p}(h(e^{2\pi i \frac{j}{m}}))|$ to sometimes be smaller than $1$. The fact that the cancellation between terms smaller than $1$ and terms greater than $1$ is nearly perfect comes from the fact that $\log \left|\tilde{p}(h(w))\right|$ is harmonic, which we make crucial use of below. 

\vspace{2mm}

\begin{lemma}\label{product}
For any $t \in [0,1]$, we have $|\tilde{p}(h(e^{2\pi i t}))| \ge \frac{1}{2}n^{-2/3}$. For any $\delta \in [0,1)$, we have $\prod_{j=J_1}^{J_2-1} |\tilde{p}(h(e^{2\pi i \frac{j+\delta}{m}}))| \le \exp(Cn^{1/3}\log^5 n)$ for some absolute $C > 0$.
\end{lemma}

\begin{proof}
Clearly both inequalities hold if $\tilde{p} = 1$, so suppose $\tilde{p}(x) = 1-x^d$ for some $1 \le d < n^{1/3}$. For the first inequality, we use $$|\tilde{p}(h(e^{2\pi i t}))| = |1-h(e^{2\pi i t})^d| \ge 1-|h(e^{2\pi i t})|^d \ge 1-(1-a)^d \ge \frac{1}{2}ad \ge \frac{1}{2}n^{-2/3}.$$ We now move on to the second inequality. Define $g(t) = 2\log|\tilde{p}(h(e^{2\pi i (t+\frac{\delta}{m})}))|$. For notational ease, we assume $\delta = 0$; the argument about to come works for all $\delta \in [0,1)$. The first inequality implies $g$ is $C^1$, so by the mean value theorem, \begin{align}\label{mvt}\left|\frac{1}{m}\sum_{j=J_1}^{J_2-1} g\left(\frac{j}{m}\right)-\int_{J_1/m}^{J_2/m} g(t)dt\right| &= \left|\sum_{j=J_1}^{J_2-1} \int_{j/m}^{(j+1)/m} \left(g(t)-g\left(\frac{j}{m}\right)\right)dt\right| \nonumber\\ &\le \sum_{j=J_1}^{J_2-1} \int_{j/m}^{(j+1)/m} \left(\max_{\frac{j}{m} \le y \le \frac{j+1}{m}} |g'(y)|\right)\frac{1}{m} dt \nonumber\\ &\le \frac{1}{m^2}\sum_{j=J_1}^{J_2-1}\max_{\frac{j}{m} \le y \le \frac{j+1}{m}} |g'(y)|. \end{align} Since $w \mapsto \log|\tilde{p}(h(w))|$ is harmonic and $\log|\tilde{p}(h(0))| = \log|\tilde{p}(0)| = 0$, we have $$\int_0^1 g(t)dt = 2\int_0^1 \log |\tilde{p}(h(e^{2\pi i t}))|dt = 0,$$ and therefore
\begin{equation} \label{harmonicapplication}
\left|\int_{J_1/m}^{J_2/m} g(t)dt\right| \le \left|\int_{0}^{J_1/m}g(t)dt\right|+\left|\int_{J_2/m}^1 g(t)dt\right|.
\end{equation}
Since $$\frac{1}{2}n^{-2/3} \le \left|\tilde{p}(h(e^{2\pi i t}))\right| \le 1$$ for each $t$, we have
\begin{equation}\label{integralbounds}
\left|\int_{0}^{J_1/m}g(t)dt\right|+\left|\int_{J_2/m}^1 g(t)dt\right| \le 2\left(\frac{J_1}{m}+(1-\frac{J_2}{m})\right)\log n \le C\frac{\log^5 n}{n^{1/3}}.
\end{equation}
By \eqref{mvt}, \eqref{harmonicapplication}, and \eqref{integralbounds}, we have $$\left|\frac{1}{m}\sum_{j=J_1}^{J_2-1} g(\frac{j}{m})\right| \le C\frac{\log^5 n}{n^{1/3}}+\frac{1}{m^2}\sum_{j=J_1}^{J_2-1} \max_{\frac{j}{m} \le t \le \frac{j+1}{m}} |g'(t)|.$$
Multiplying through by $m$, changing $C$ slightly, and exponentiating, we obtain 
\begin{equation}\label{partialproductbound}
\prod_{j=J_1}^{J_2-1} \left|\tilde{p}(h(e^{2\pi i \frac{j}{m}}))\right|^2 \le \exp\left(Cn^{1/3}\log^5 n + \frac{1}{m}\sum_{j=J_1}^{J_2-1} \max_{\frac{j}{m} \le t \le \frac{j+1}{m}} |g'(t)|\right).
\end{equation}
Note $$g'(t_0) = \frac{\frac{\partial}{\partial t}\Big[|\tilde{p}(h(e^{2\pi i t}))|^2\Big] \Big|_{t=t_0}}{|\tilde{p}(h(e^{2\pi i t_0}))|^2}.$$ 

\noindent We first show $$\frac{\partial}{\partial t}\Big[|\tilde{p}(h(e^{2\pi i t}))|^2\Big] \Big|_{t=t_0} \le 100d$$ for each $t_0 \in [0,1]$. We start by noting $$\Big|\tilde{p}(h(e^{2\pi i t}))\Big|^2 = 1+(1-a)^{2d}\left(\left|\sum_{j=1}^r d_j e^{2\pi i tj}\right|^2\right)^d-2\Real\left[\left((1-a)\sum_{j=1}^r d_j e^{2\pi i tj}\right)^d\right].$$ Let $$f_1(t) = (1-a)^{2d}\left(\left|\sum_{j=1}^r d_j e^{2\pi i tj}\right|^2\right)^d.$$ Then, \begin{align*}f_1'(t) &= (1-a)^{2d}d\left(\left|\sum_{j=1}^r d_j e^{2\pi i tj}\right|^2\right)^{d-1}\frac{\partial}{\partial t}\left[\left|\sum_{j=1}^r d_j e^{2\pi i tj}\right|^2\right] \\ &= (1-a)^{2d}d\left(\left|\sum_{j=1}^r d_j e^{2\pi i tj}\right|^2\right)^{d-1}\sum_{1 \le j_1,j_2 \le r} d_{j_1}d_{j_2}2\pi i(j_1-j_2)e^{2\pi i (j_1-j_2)t}.\end{align*} Since $\sum_{j=1}^r d_j = 1$, we therefore have \begin{align*}|f_1'(t)| &\le 2\pi d\sum_{1 \le j_1,j_2 \le r} \lambda_a^2\frac{j_1+j_2}{j_1^2j_2^2\log^2(j_1+3)\log^2(j_2+3)} \\ &= 4\pi d\left(\sum_{j_1=1}^r \frac{\lambda_a}{j_1\log^2(j_1+3)}\right)\left(\sum_{j_2=1}^r \frac{\lambda_a}{j_2^2\log^2(j_2+3)}\right) \\ &\le 50d. \end{align*} Now, let $$f_2(t) = -2\Real\left[\left((1-a)\sum_{j=1}^r d_j e^{2\pi i tj}\right)^d\right]$$ and note \begin{align*} f_2'(t) &= \frac{\partial}{\partial t}\left[-2(1-a)^d\sum_{1 \le j_1,\dots,j_d \le r} d_{j_1}\dots d_{j_d} \cos(2\pi t(j_1+\dots+j_d))\right] \\ &= 4\pi (1-a)^d \sum_{1 \le j_1,\dots,j_d \le r} d_{j_1}\dots d_{j_d}(j_1+\dots+j_d)\sin(2\pi t(j_1+\dots+j_d)),\end{align*} yielding \begin{align*} |f_2'(t)| &\le 4\pi \sum_{1 \le j_1,\dots,j_d \le r}\lambda_a^d \frac{j_1+\dots+j_d}{j_1^2\dots j_d^2 \log^2(j_1+3)\dots \log^2(j_d+3)} \\ &= 4\pi d\left(\sum_{j_1=1}^r \frac{\lambda_a}{j_1\log^2(j_1+3)}\right)\left(\sum_{j=1}^r \frac{\lambda_a}{j^2\log^2(j+3)}\right)^{d-1} \\ &\le 50d. \end{align*} We have thus shown $$\frac{\partial}{\partial t}\Big[|\tilde{p}(h(e^{2\pi i t}))|^2\Big] \Big|_{t=t_0} \le 100d$$ for each $t_0 \in [0,1]$. 

\vs

\noindent Recall $$|\tilde{p}(h(e^{2\pi i t}))| = |1-h(e^{2\pi i t})^d| \ge 1-|h(e^{2\pi i t})|^d.$$ 

\vspace{2.5mm}

\noindent For $j \in [J_1,J_2] \sub [C_6am,(1-C_6a)m]$, we use $$|h(e^{2\pi i \frac{j}{m}})| \le 1-c_5\frac{\min(\frac{j}{m},1-\frac{j}{m})}{\log^2 n}$$ to obtain $$\frac{1}{m}\sum_{j=J_1}^{J_2-1}\max_{\frac{j}{m} \le t \le \frac{j+1}{m}} |g'(t)| \le \frac{1}{m}\sum_{j=J_1}^{J_2-1} \frac{100d}{\left(1-(1-c_5\frac{\min(\frac{j}{m},1-\frac{j}{m})}{\log^2 n})\strut^d\right)^2}.$$ Up to a factor of $2$, we may deal only with $j \in [J_1,\frac{m}{2}]$. Let $J_* = c_5^{-1}d^{-1}m\log^2 n$. Note that $j \le J_*$ implies $c_5\frac{j}{m\log^2 n} \le d^{-1}$ and $j \ge J_*$ implies $c_5\frac{j}{m\log^2n} \ge d^{-1}$. Thus, using $(1-x)^d \le 1-\frac{1}{2}xd$ for $x \le \frac{1}{d}$, we have \begin{align}\label{range1}\frac{1}{m}\sum_{j=J_1}^{\min(J_*,\frac{m}{2})} \frac{100d}{\left(1-(1-c_5\frac{j}{m\log^2 n})\strut^d\right)^2} &\le \frac{100d}{m}\sum_{j=J_1}^{\min(J_*,\frac{m}{2})} \frac{1}{\left(\frac{1}{2}c_5\frac{j}{m\log^2 n}d\right)^2} \nonumber \\ &= \frac{400m\log^4 n}{c_5^2 d}\sum_{j=J_1}^{\min(J_*,\frac{m}{2})} \frac{1}{j^2} \nonumber \\ &\le \frac{400m\log^4n}{c_5^2d}\frac{2}{J_1} \nonumber \\ &\le Cn^{1/3}. \end{align} Finally, since there is some $c > 0$ such that $(1-x)^l \le 1-c$ for all $l \in \N$ and $x \in [l^{-1},1]$, using the notation $\sum_{i=a}^b x_i = 0$ if $a > b$, we see \begin{align}\label{range2} \frac{1}{m}\sum_{j = \min(J_*,\frac{m}{2})+1}^{m/2} \frac{100d}{\left(1-(1-c_5\frac{j}{m\log^2 n})\strut^d\right)^2} &\le \frac{100d}{m}\sum_{j=\min(J_*,\frac{m}{2})+1}^{m/2} c^{-2} \nonumber \\ &\le Cd \nonumber \\ &\le Cn^{1/3}. \end{align} Combining \eqref{range1} and \eqref{range2}, we obtain $$\frac{1}{m}\sum_{j=J_1}^{J_2-1} \max_{\frac{j}{m} \le \frac{j+1}{m}} |g'(t)| \le Cn^{1/3}.$$ Plugging this upper bound into \eqref{partialproductbound} yields the desired result.
\end{proof}

\vspace{1mm}

\begin{proof}[Proof of Proposition \ref{mmp}]
Define $g(z) = \prod_{j=0}^{m-1} p(h(e^{2\pi i\frac{j}{m}}z))$. Fix $z \in \partial \mathbb{D}$; say $z = e^{2\pi i(\frac{j_0}{m}+\delta)}$ for some $j_0 \in \{0,\dots,m-1\}$ and $\delta \in [0,\frac{1}{m})$. For ease of notation, we assume $j_0 = 0$; the argument about to come is to any $j_0$. Then, $e^{2\pi i\frac{j}{m}}z$ is in $\{e^{2\pi it} : -c_4a \le t < c_4a\}$ if $j \in \{0,m-1\}$. Therefore, Lemma \ref{product} followed by the maximum modulus principle ($p$ is analytic) imply
\begin{align}\label{mmpeqn}
|g(z)| &\le \left(\max_{w \in \widetilde{E}_a^\circ} |p(w)|\right)^2\prod_{j \not \in \{0,m-1\}} |p(h(e^{2\pi i\frac{j}{m}}z))| \nonumber \\ &\le \left(\max_{w \in \widetilde{E}_a} |p(w)|\right)^2\prod_{j \not \in \{0,m-1\}} |p(h(e^{2\pi i\frac{j}{m}}z))|.
\end{align}
Let $I = [J_1,J_2-1]\cap \Z$. For $j \not \in I$, using the bound $|p(w)| \le \frac{1}{1-|w|}$ for each $w \in \partial \mathbb{D}$, we see $$|p(h(e^{2\pi i\frac{j}{m}}z))| \le \frac{1}{1-|h(e^{2\pi i\frac{j}{m}}z)|} \le \frac{1}{1-(1-a)} = n^{2/3},$$ thereby obtaining
\begin{equation}\label{jnotinB}
\prod_{j \not \in I \cup \{0,m-1\}} |p(h(e^{2\pi i\frac{j}{m}}z))| \le (n^{2/3})^{(J_1-1)+(m-J_2+1)} \le (n^{2/3})^{Cn^{1/3}\log^4 n} \le e^{Cn^{1/3}\log^5n}.
\end{equation}
Now, for $j \in I$, since $$|h(e^{2\pi i \frac{j}{m}}z)| \le 1-c_5\frac{\min\left(\frac{j}{m}+\delta,1-(\frac{j}{m}+\delta)\right)}{\log^2n} \le 1-c'n^{-1/3}\log^2n,$$ we have $$\left|p\hspace{-.5mm}\left(h(e^{2\pi i \frac{j}{m}z})\right)-\tilde{p}\hspace{-.5mm}\left(h(e^{2\pi i \frac{j}{m}z})\right)\right| \le ne^{-c'\log^2 n} \le e^{-c\log^2n}.$$ Therefore,
\begin{equation}\label{pbyptilde}
\prod_{j \in I} |p(h(e^{2\pi i\frac{j}{m}}z))| \le \prod_{j \in I} \left(|\tilde{p}(h(e^{2\pi i\frac{j}{m}}z))|+e^{-c\log^2 n}\right).
\end{equation}
By both parts of Lemma \ref{product}, we obtain \begin{align}\label{jinB}\prod_{j \in I} \left(|\tilde{p}(h(e^{2\pi i\frac{j}{m}}z))|+e^{-c\log^2 n}\right) &= \sum_{I' \sub I} \left(\prod_{j \in I\setminus I'} |\tilde{p}(h(e^{2\pi i\frac{j}{m}}z))|\right)e^{-c(\log^2 n)|I'|} \nonumber \\ &= \medmath{\sum_{I' \sub I} \left(\prod_{j \in I} |\tilde{p}(h(e^{2\pi i\frac{j}{m}}z))|\right)\left(\prod_{j \in I'} |\tilde{p}(h(e^{2\pi i\frac{j}{m}}z))|\right)^{-1}e^{-c(\log^2 n) |I'|}} \nonumber \\ &\le e^{Cn^{1/3}\log^5n} \sum_{I' \sub I} (2n^{2/3})^{|I'|}e^{-c(\log^2n) |I'|} \nonumber \\ &\le e^{Cn^{1/3}\log^5 n}\sum_{I' \sub I} e^{-c'(\log^2 n) |I'|} \nonumber \\ &\le e^{Cn^{1/3}\log^5 n}\sum_{k=0}^{|I|} {|I| \choose k} e^{-c'k\log^2 n} \nonumber \\ &\le 2e^{Cn^{1/3}\log^5 n}. \end{align} Combining \eqref{mmpeqn}, \eqref{jnotinB}, \eqref{pbyptilde}, and \eqref{jinB}, we've shown $$|g(z)| \le \left(\max_{z \in \widetilde{E}_a} |p(z)|\right)^2 e^{Cn^{1/3}\log^5 n}.$$ As this holds for all $z \in \partial \mathbb{D}$, we have $$\max_{z \in \partial \mathbb{D}} |g(z)| \le  \left(\max_{z \in \widetilde{E}_a} |p(z)|\right)^2 e^{Cn^{1/3}\log^5 n}.$$ To finish, note that $|g(0)| = |p(h(0))|^m = |p(0)|^m = 1$, so, as $g$ is clearly analytic, the maximum modulus principle implies $\max_{z \in \partial \mathbb{D}} |g(z)| \ge 1$.
\end{proof}

\vspace{1mm}

\section{Tightness of our methods}

In this section, we prove the following, showing that our methods cannot be pushed further in some sense. We denote $\{0,1\}^{\le p} := \cup_{j=1}^p \{0,1\}^j$. 

\begin{proposition}\label{sparseexample}
For all $n$ large, there are distinct strings $x,y \in \{0,1\}^n$ such that for all $p \le \frac{1}{10}n^{1/3}$, $i \in [p]_0$, and $w \in \{0,1\}^{\le p}$, it holds that $|\emph{\text{pos}}_w(x)_{i,p}| = |\emph{\text{pos}}_w(y)_{i,p}|$. 
\end{proposition}

We begin by showing Theorem \ref{diffprime} is tight, via a standard pigeonhole argument that has been used in a variety of other papers. 

\begin{proposition}\label{tight}
For all $n$ large, there are distinct $n^{1/3}$-separated subsets $A,B$ of $[n]$ such that $|A_{i,p}| = |B_{i,p}|$ for all $p \le cn^{1/3}\log^{1/2}n$ and all $i \in [p]_0$. 
\end{proposition}

\begin{proof}
Let $\Sigma$ denote the collection of subsets $A \sub [n]$ that have at most one number from each of the intervals $[1,n^{1/3}],[2n^{1/3},3n^{1/3}],[4n^{1/3},5n^{1/3}],\dots$. Note $|\Sigma| \ge (n^{1/3})^{\frac{1}{3}n^{2/3}} = e^{\frac{1}{9}n^{2/3}\log n}$. On the other hand, for any $A \sub [n]$, the number of possible tuples $(|A_{i,p}|)_{\substack{p \le k \\ i \in [p]_0}}$ is at most $\prod_{p \le k} n^p \le e^{\frac{k^2}{\log k}\log n}$. Taking $k = cn^{1/3}\log^{1/2}n$ yields $\frac{k^2}{\log k}\log n < \frac{1}{9}n^{2/3}\log n$, meaning there are distinct $A,B \in \Sigma$ with the same tuple, i.e. $|A_{i,p}| = |B_{i,p}|$ for all $p \le k$ and $i \in [p]_0$. As $A,B$ are $n^{1/3}$-separated, the proof is complete. 
\end{proof}

\vspace{1mm}

\begin{proof}[Proof of Proposition \ref{sparseexample}]
For a large $n$, let $A,B \sub [n/2]$ be the sets guaranteed by Proposition \ref{tight}. Let $x = (1_A(j-\frac{n}{4}))_{j=1}^n, y = (1_B(j-\frac{n}{4}))_{j=1}^n \in \{0,1\}^n$ be the strings with $1$s at indices in $A$ and $B$ then padded at the beginning and end by $0$s. Fix $p \le \frac{1}{10}n^{1/3}$ and $i \in [p]_0$. Since $A,B$ are $\frac{1}{10}n^{1/3}$-separated, we have $|\pos_w(x)_{i,p}| = |\pos_w(y)_{i,p}| = 0$ for all $w \in \{0,1\}^{\le p}$ with at least two $1$s. Since $$\pos_{0^l}(x) = [n-l+1]\setminus\sqcup_{s=0}^{l-1} \pos_{0^s10^{l-1-s}}(x),$$ it suffices to show $|\pos_w(x)_{i,p}| = |\pos_w(y)_{i,p}|$ for all $w \in \{0,1\}^{\le p}$ with exactly one $1$. Fix such a $w$; say $w = 0^s10^{l-1-s}$ for some $l \le p$ and $s \in \{0,\dots,l-1\}$. Then, due to the padding preventing boundary issues, $\pos_w(x) = \{j : x_{j+s} = 1\} = \{j : 1_A(j+s-\frac{n}{4}) = 1\} = A-s+\frac{n}{4}$ and thus $|\pos_w(x)_{i,p}| = |A_{i+s-\frac{n}{4},p}|$. Similarly, $|\pos_w(y)_{i,p}| = |B_{i+s-\frac{n}{4},p}|$. Since $p \le c(n/2)^{1/3}\log^{1/2}(n/2)$, the proof is complete. 
\end{proof}

\vspace{1mm}

\section{Acknowledgments} 

I would like to thank my advisor Ben Green for several helpful comments on the readability of the paper and Noah Golowich for pointing out a flaw in a claimed generalization of one of the propositions in a previous version of the paper.

\end{document}